\documentclass[a4paper]{article}

\usepackage{amsmath}
\usepackage{amssymb}
\usepackage{a4}

\usepackage{graphicx}
\usepackage[all]{xy}

\def \defn#1{{\em{#1}}}
\def \sf{\mathop{\rm SF}\nolimits}
\def \gf{\mathop{\rm GF}\nolimits}
\def \mf{\mathop{\rm MF}\nolimits}
\def \tmf{\mathop{\rm TMF}\nolimits}
\def \dis{\mathop{\rm Dis}\nolimits}
\def \cl#1{\mathop{{\rm {Cl}}\left(#1\right)}\nolimits}
\def \interior#1{\mathop{{\rm {Int}}\left(#1\right)}\nolimits}
\def \notemp{\mathop{\beta}\nolimits}

\def \dim{\mathop{\rm dim}\nolimits}

\def \C{{\mathbb C}}

\def \AA{{\cal A}}

\def \LL{{\cal L}}

\def \CC{{\cal C}}

\def \PP{{\cal P}}
\def \SS{{\cal S}}

\newtheorem{theorem}{Theorem}[section]
\newtheorem{corollary}[theorem]{Corollary}
\newtheorem{proposition}[theorem]{Proposition}
\newtheorem{lemma}[theorem]{Lemma}
\newtheorem{remark}[theorem]{Remark}
\newtheorem{remarks}[theorem]{Remarks}

\newtheorem{examples}[theorem]{Examples}
\newtheorem{definition}[theorem]{Definition}

\newenvironment{proof}
 {\begin{trivlist} \item[\hskip \labelsep {\bf Proof.}]}
 {\hfill$\Box$\end{trivlist}}

\begin{document}
\renewcommand{\theenumi}{(\roman{enumi})}
\normalsize

\title{Equisingularity and The Euler Characteristic of a Milnor Fibre} 
\author{Kevin Houston \\
School of Mathematics \\ University of Leeds \\ Leeds, LS2 9JT, U.K. \\
e-mail: k.houston@leeds.ac.uk \\
http://www.maths.leeds.ac.uk/$\sim $khouston/
}
\date{\today }
\maketitle
\begin{abstract}
We study the Euler characteristic of the Milnor fibre of a hypersurface singularity.
This invariant is given in terms of the Euler characteristic of a fibre in between the original singularity and its Milnor fibre and in terms of the Euler characteristics associated to strata of the in-between fibre.

From this we can deduce a result of Massey and Siersma regarding singularities with a one-dimensional critical locus.
The result is also applied to the study of equisingularity.
The famous Brian\c{c}on-Speder-Teissier result states that a family of isolated hypersurface singularities is equisingular if and only if its $\mu ^*$-sequence is constant. We show that if a similar sequence for a family of corank 1 complex analytic mappings from $n$-space to $n+1$-space is constant, then the image of the family of mappings is equisingular. For families of corank 1 maps from $3$-space to $4$-space we show that the converse is true also.


\end{abstract}

\section{Introduction}
The Euler characteristic is the most basic and, possibly, most powerful invariant in topology. 
It is well known to be crude but extremely effective.
A powerful invariant in the study of complex analytic hypersurface singularities is the Milnor fibre. In this paper the main results on equisingularity follow from careful study of the Euler characteristic of a certain Milnor fibre.

For an isolated singularity Milnor showed that the Milnor fibre is homotopically equivalent to a bouquet of spheres and hence the number of these spheres is called the Milnor number. This number has been incredibly successful in the study of singularities as it can be defined by topological and algebraic means.
 In the case of non-isolated singularities not many general theorems are known since the Milnor fibre may not be a bouquet of spheres and a good link between the algebra and topology has yet to be found. See \cite{num-control} for the state of the art. Interestingly, two of the earliest general results were published in the same conference proceedings. One is the Kato-Matsumoto result that relates the vanishing of homotopy groups of the Milnor fibre to the dimension of the singular set of the singularity, \cite{km}. The other, by Sakamoto, expresses the homotopy type of the Milnor fibre of the direct sum or direct product of two singularities in terms of Milnor fibres of the constituent singularities, \cite{sakamoto}.

Even more surprisingly, a number of the results that are known for particular examples of singularities, see \cite{hn, nemethi2,nemethi} and references therein, are for the homotopy type of the Milnor fibre rather than for its Euler characteristic.  One would assume it is easier to prove results for the latter as it is simpler.

In this paper we apply some general theorems on the Euler characteristic of the Milnor fibre of a non-isolated hypersurface singularity. Our method is to perturb the singularity to produce a fibre that is in some sense in between the original singularity and the Milnor fibre. Using the topology of this fibre and the topology of its stratification we shall be able to calculate the Euler characteristic of the Milnor fibre. The proof involves making small changes to results in the existing literature on vanishing cycles.

Taking an in-between fibre is entirely natural.
For example, the singularities of the in-between fibre may be a smoothing of the singularities of the original. This can be seen in the Generalized Zariski examples in Section~\ref{sec:gen-zariski}. It was the main method used by Siersma and his students, see the references in \cite{hn, nemethi2,nemethi}, where the singular set of the hypersurface is assumed to be an isolated complete intersection singularity. The in-between fibre's singular set is the Milnor fibre of this ICIS. 

Another example of an in-between fibre is of great importance in the study of equisingularity. Here the hypersurface singularity is the image of a map $f$ with an isolated instability. We can perturb $f$ so that the image of the perturbation has only stable singularities. This image is called the disentanglement of $f$ and plays the part of the in-between fibre. 
In general, the disentanglement has non-isolated singularities (as does $f$) and is homotopically equivalent to a wedge of spheres.

An obvious and interesting question is what is the relationship between the topology of the disentanglement of $f$ and the Milnor fibre of the image of $f$. Certainly, David Mond and the author discussed this as long ago as 1993. An answer was not pressing as there seemed to be no obvious problems that could be solved with it. However, recent research, see \cite{diswe} and \cite{cntocn+1}, indicates that the precise relationship is extremely useful in the study of equisingularity. By providing such a relationship we prove Theorem~\ref{mainequithm} on equisingularity of mappings.

The paper is arranged as follows.
Section~\ref{sec:notation} defines the notation and basic concepts and contains the central result from which all the others follow. A generalization of the Milnor fibration is introduced. In it, there is a special fibre (which is the in-between fibre in applications) and a general fibre (which is homeomorphic to the Milnor fibre in applications). Theorem~\ref{maineulerthm} relates the difference of the Euler characteristics of the general fibre and special fibre to the Euler characteristics of strata in the special fibre and the Euler characteristics of the Milnor fibres of singularities transverse to these strata. 

In the next section we see how a generalization of a result of Massey and Siersma from \cite{ms} on hypersurfaces with a $1$-dimensional critical locus can be deduced from the main theorem. We also see how the theorem can be used to study the topology of composed singularities.

Section~\ref{sec:dis} contains the main result of interest that relates the Euler characteristic of the disentanglement of a complex analytic map-germ $f:(\C ^n ,0)\to (\C ^p,0)$, with $n\geq p-1$ and an isolated instability, to the Milnor fibre of the function that defines the discriminant of the map in terms of a stratification of the discriminant and its transverse Milnor fibres.
Theorem~\ref{firstlinsum} shows how to calculate the Euler characteristic of the closure of a stratum in the image of the disentanglement using the Milnor numbers of the multiple point spaces of $f$.

In the next section we apply this to the special case where $n=3$, $p=4$ and the map has corank $1$. Here we can describe precisely the topology of closure of strata using the Milnor numbers of the multiple point spaces of $f$. The transverse Milnor fibres have Euler characteristic $0$ except in the case of one stratum.
In the following section we see a very particular example of a map 
$f:(\C ^n ,0)\to (\C ^n,0)$ such that, outside of the origin, the singularities of the discriminant are at worst cuspidal. 

In Section~\ref{sec:equising} the results are applied to equisingularity. That is we unfold a map with an isolated instability and attempt to stratify this unfolding so that the parameter axis in the image is a stratum. The main result here is that using Damon's higher multiplicities we can achieve a Brianc\c{c}on--Speder-Teissier type result in that constancy of a $\mu ^*$-sequence implies that the axis is a stratum in a Whitney stratification.

The final section finishes with conclusions and some thoughts on future work. In particular, the converse of the above theorem relies on the value of the Euler characteristic of the Milnor fibre of the image of a corank $1$ map. Finding this value is an open problem and would be a good test of the mathematical machinery developed for non-isolated singularities so far.

The author began this paper while he was a visitor at 
Universit\'e des Sciences et Technologies de Lille, France. He is grateful to Mihai Tib\u{a}r and
Arnaud Bodin for their generous hospitality and USTL for funding.

\section{Quasi-Milnor fibrations}
\label{sec:notation}

We shall now define a generalization of the Milnor fibration.
Suppose that $h:U\to \C$ is a complex analytic map from the open set $U\subseteq \C ^n$ such that $n\geq 1$ and $0\in U$.
(In applications $h^{-1}(0)$ within a small ball will define a fibre in between a singularity and its Milnor fibre.)

Let $\SS '$ be a Whitney stratification of $h^{-1}(0)$ in $U$ so that each stratum is connected.
For any subset $Z$ of $U$ we define $\overline{Z}$ to be the closure of $Z$ in $U$.

Let $B_\epsilon (0)$ denote the open ball of radius $\epsilon $ centred at $0$ and let $S_\epsilon (0)$ denote the boundary of its closure, the sphere of radius $\epsilon $ centred at $0$. Let $D_\eta $ be the open disc in $\C $ centred at $0$ with radius $\eta $.

\begin{definition}
We say that $h$ is a {\defn{quasi-Milnor fibration}} if there exists $\epsilon >0$ and $\eta >0$ so that  
\begin{enumerate}
\item $B_\epsilon (0) \subset U$,
\item $S_\epsilon (0) \subset U$,
\item $S_\epsilon (0)$ is transverse to all strata of $h^{-1}(0)$, 
\item $h|\overline{B_\epsilon(0)}\cap h^{-1}(D_\eta ) \to D_\eta $ is a proper map,
\item $h|B_\epsilon(0)\cap h^{-1}(D_\eta ) \to D_\eta $ and
$h|S_\epsilon(0)\cap h^{-1}(D_\eta ) \to D_\eta $ have only a finite number of critical values in $D_\eta $, and the set of those arising from $S_\epsilon (0)$ is contained in the set of those of $B_\epsilon $,
\item the induced Whitney stratification of $\overline{B_\epsilon (0)\cap h^{-1}(0)}$ has only a finite number of strata.
\end{enumerate}
\end{definition}

\begin{definition}
The {\defn{general fibre}} of $h$, denoted $\gf $, is the set $h^{-1}(t)\cap B_\epsilon (0)$ where $t\in D_\eta $ is not a critical value of $h|B_\epsilon(0)\cap h^{-1}(D_\eta )$. 

The {\defn{special fibre}} of $h$, denoted $\sf $, is the set $h^{-1}(0)\cap B_\epsilon (0)$. 
\end{definition}
The main result of this section allows a comparison of the Euler characteristic of the special and general fibres.

Note that is not necessary that $0\in \sf $ which is the case for a Milnor fibration. Furthermore, note that by Sard's theorem and Thom's First Isotopy Lemma we can deduce that the general fibre is a manifold which is well-defined up to homeomorphism.

If in addition to the above conditions we assume that we have transversality of $S_{\epsilon '}(0)$ for all $0<\epsilon ' <\epsilon $ (and so $0\in \sf$), then we get that $\gf $ is the Milnor fibre of $h$ and $\sf$ is the Milnor cone (i.e., is homeomorphic to the cone over the real link of the singularity at the origin).
In our applications we shall not have this extra condition. The idea is that $\gf$, the general fibre, is homeomorphic to the Milnor fibre of a function, $f$ say, $h$ is some perturbation of $f$ with nice properties and $\sf$ is some set, again with nice properties. In some sense, $\sf $ lies in between the Milnor cone and the Milnor fibre of $f$ - a non-singular set - in terms of the complexity of its singularities.

If $Z$ is a subset of $B_\epsilon (0)$, then define $\cl{Z}$ to be the closure of $Z$ in $B_\epsilon (0)$. We define $\interior{Z}$ to be the interior of $Z$ in $B_\epsilon (0)$ and define $\partial Z$ to be $\overline{\cl{Z}}\backslash \cl{Z}$. Thus $\partial Z$ is the boundary of $Z$ that lies in the sphere $S_\epsilon (0)$.
If $Z$ is closed 
in $B_\epsilon (0)$, then $\overline{Z} = Z \cup \partial Z$.

\begin{definition}
We shall say that a Whitney stratified space $Z\subset U $ is {\em{collared}} in $\overline{B_\epsilon (0)}$ if the inclusion $Z\cap B_\epsilon (0)$ into $Z\cap \overline{B_{\epsilon }(0)}$ is a homotopy equivalence.
\end{definition}

The condition that $S_\epsilon (0)$ is transverse to $h^{-1}(0)$ means that the general fibre is in fact collared in $\overline{B_\epsilon (0)}$. 
\begin{examples}
The sets $\sf $, $\gf $ and closure of strata of $\SS $ in $\overline{B_\epsilon (0)}$ are all collared.  

Thus, in particular, as is well-known, the Milnor fibre and Milnor cone of any complete intersection singularity are collared.
\end{examples}

For convenience we say that $Z$ is collared if $\partial Z=\emptyset $. For example $Z$ could be a point.

For a quasi-Milnor fibration let $\SS $ be the stratification of $B_\epsilon (0)\cap h^{-1}(0)$ induced from $\SS '$. 

Let $X$ be a stratum in $\SS $ and let $x$ be any point of $X$. 
For a manifold $T$ transverse to $X$ with $T\cap X$ a single point let $\tmf _{X,x}$ be the open Milnor fibre of $h|T\cap B_\epsilon (0)$ at $x\in X$. Since the topological type of a neighbourhood of a stratum is independent of the point in the stratum and by a result of L\^{e} (\cite{le}) the homotopy type of this set is independent of the $x$ chosen in $X$. 
\begin{definition}
The {\defn{transverse Milnor fibre of $X$}}, denoted $\tmf _X$, is the homotopy class of $\tmf _{X,x}$.
\end{definition}
Note that the transverse Milnor fibre and its closure are homotopically equivalent as the closure in $U$ is collared in the sense above.

Let $\chi (Z)$ denote the Euler characteristic of the set $Z$, $\widetilde{\chi}(Z)$ its reduced Euler characteristic and $\chi _c(Z)$ denote the Euler characteristic with compact support of $Z$. Note that $\chi _c(Z)$ coincides with $\chi (Z)$ when $Z$ is compact. Note also that $\chi _c$ is additive in locally compact spaces and so will be additive in our examples. By additive we mean that if $Z=Z_1 \cup Z_2$ where $Z_2$ is closed in $Z$ and $Z_1\cap Z_2=\emptyset $, then $\chi _c \left( Z \right) = \chi _c \left( Z_1 \right) + \chi _c \left( Z_2 \right)$. See, for example, \cite{iversen} p185.


\label{sec:comparison}
The next lemma is probably well-known but I cannot locate a suitable reference.
\begin{lemma}
\label{collared}
Suppose that $Z$ is the closure in $B_\epsilon (0)$ of a stratum in $\SS $.
Then $\chi _c(Z) = \chi (Z) $.
\end{lemma}
\begin{proof}
We have
\begin{eqnarray*}
\chi (Z) &=& \chi \left( \overline{Z} \right) {\text{ as $Z$ is collared,}} \\
&=& \chi _c   \left( \overline{Z} \right) {\text{ as $\overline{Z}$ is compact,}} \\
&=&\chi _c   \left( Z \right)  + \chi _c   \left( \partial Z \right)   {\text{ as $\chi _c$ is additive,}} \\
&=&\chi _c   \left( Z \right)  + \chi   \left( \partial Z \right)   {\text{ as $\partial Z$ is compact,}} \\
&=&\chi _c   \left( Z \right)  . 
\end{eqnarray*}
The last equality follows from Sullivan's well-known result in \cite{sullivan} that the Euler characteristic of a stratified space with only odd-dimensional strata is zero.
\end{proof}

\begin{lemma}
\label{lemma2}
Let $h$ be a quasi-Milnor fibration. For all $X\in \SS $ we have
\[
\chi \left( \cl{X}, \cl{X}\backslash X \right)
=\chi \left( \overline{X} , \cup_{Y<X} \overline{Y} \right) 
\]
where as usual $\cl{X}$ denotes closure in the open ball $B_\epsilon (0)$.
\end{lemma}
\begin{proof}
First we note that $\cl X \backslash X $ is just $\cup_{Y<X }Y$ which is the same set as $ \cup_{Y<X} \overline{Y}$. Both sets are collared so $\chi \left( \cl{X} \right)
=\chi \left( \overline{X} \right) $ and $\chi \left( \cl{X}\backslash X \right)
=\chi \left(  \cup_{Y<X} \overline{Y} \right) $. 
\end{proof}

We now come to the main theorem in this section. 
\begin{theorem}
\label{maineulerthm}
For a quasi-Milnor fibration the Euler characteristic of the general fibre and special fibre are related by the following:
\[
\chi (\gf ) - \chi (\sf ) = \sum_{X\in \SS }  \widetilde{\chi } (\tmf _X) \chi \left( \cl{X}, \cl{X}\backslash X \right) .
\]

\end{theorem}
\begin{proof}
We can use the theory of vanishing cycles. A good reference for this is \cite{dimca}. See also Appendix B of \cite{num-control}.

Let $\varphi _h \C ^\bullet _{B_\epsilon (0)}$ be the sheaf of vanishing cycles of $h$ on $B_\epsilon (0)$. This is well-known to be constructible (see Appendix B of \cite{num-control}).

We can now use a compact support version of Proposition~6.2.1 of \cite{dimca}. The result there is stated for a proper map, which our $h|B_\epsilon (0)$ certainly is not, but the conclusion still holds if we drop the proper assumption and use Euler characteristic with compact support. That is, we have
\[
\chi _c (\gf ) - \chi _c(\sf ) = \chi _c \left( \sf , \varphi _h \C ^\bullet _{B_\epsilon (0)} \right)  .
\]
Now using the compact support version of Lemma~6.2.7 of \cite{dimca} and the paragraph preceding it we have
\[
\chi _c \left( \sf , \varphi _h \C ^\bullet _{B_\epsilon (0)} \right) 
= \sum_{X\in \SS } \chi _c(X) \widetilde{\chi }  \left( \mf _{h,x_X} \right) .
\]
where $\mf _{h,x_X}$ is the Milnor fibre of $h$ at the point $x_X\in X$. Since $\sf $ is a product along $X$ and using the result of L\^e in \cite{le}, we have that $\mf _{h,x_X}$ is homotopically equivalent to $\tmf _X $ 
so $\widetilde{\chi } \left( \mf _{h,x_X} \right) = \widetilde{\chi } \left( \tmf _X \right) $.

Since $\gf $ and $\sf $ are collared we have by Lemma~\ref{collared} that 
\[
\chi _c (\gf ) - \chi _c (\sf) = 
\chi (\gf ) - \chi (\sf) .
\]
This just leaves to show that $\chi _c (X) = \chi \left( \cl{X}, \cl{X}\backslash X \right) $.
We have
\begin{eqnarray*}
\chi _c \left( \cl{X}\right) &=&\chi _c \left( X \cup \left( \cl{X} \backslash X  \right) \right) \\
&=&\chi _c \left( X \right) +  \chi _c \left( \cl{X} \backslash X  \right) , {\text{ by additivity of }}\chi _c, \\
\chi  \left( \cl{X}\right) &=&\chi _c \left( X \right) +  \chi \left( \cl{X} \backslash X \right) , {\text{ by Lemma~\ref{collared},}} \\
\chi  \left( \cl{X}\right) - \chi \left( \cl{X} \backslash X \right) &=&\chi _c \left( X \right)   \\
\chi  \left( \cl{X}, \cl{X} \backslash X \right) &=& \chi _c \left( X \right).
\end{eqnarray*}
\end{proof}

\begin{remark}
For a reduced defining equation $h$, if we consider the top-dimensional strata in $\SS $, then $\tmf _X$ is homotopically equivalent to a point. Hence, $\widetilde{\chi }\left(\tmf _X\right)=0$ and we can ignore the top-dimensional strata in the above formula. 
\end{remark}

\section{Deducing the Massey--Siersma result and composing singularities}
\label{sec:gen-zariski}

Theorem~\ref{maineulerthm} can be seen as a generalization of one of the main theorems in \cite{ms}.
After showing this, we shall apply the theorem to investigate composed singularities as they do in \cite{ms}.

For a function $f$ let $\Sigma f$ denote its critical set, i.e., the points where it is not a submersion. We denote by $V(f)$ the zero-set of $f$.

Let $(z_0,z_1, \dots , z_n)$ be coordinates for $\C ^{n+1}$. Suppose that $f_0:\C ^{n+1}\to \C$ is a reduced complex analytic map such that $f_0(0)=0$, $\Sigma f_0$ is $1$-dimensional and $\Sigma \left( f_0|V(z_0)\right) $ is $0$-dimensional.

Now assume that $f_s:\C ^{n+1}\to \C$ is a family of analytic maps that deforms $f_0$.
Let $f(z_0,z_1, \dots , z_n,s)=f_s(z_0,z_1, \dots , z_n)$. 
\begin{definition}
The family $f_s$ is an {\em{equi-transversal deformation}} of $f_0$ if for all components $\nu $ of $\Sigma f_0$ through the origin in $\C ^{n+1}$ there exists a unique component $\nu '$ of $\Sigma f$ containing $\nu $ and moreover $\stackrel{\circ}{\mu}_\nu = \stackrel{\circ}{\mu}_{\nu '}$. Here $\stackrel{\circ}{\mu}_\nu$ denotes the Milnor number of the transverse Milnor fibre along the component $\nu $. A similar definition is made for $\nu '$. 
\end{definition}

If $f_s$ is an equi-transversal deformation of $f_0$, then there exists $\epsilon >0$ and $a>0$ so that, defining $h(x)=f_a(x)$ we produce a quasi-Milnor fibration with the general fibre of $h$ homeomorphic to the Milnor fibre of $f_0$ and the special fibre equal to $f_a^{-1}(0)\cap B_\epsilon $. This is effectively proved in Lemmas 2.1 and 2.2 of \cite{ms}, but see also \cite{siersma-wark}.

Massey and Siersma state and prove the following theorem. Readers unfamiliar with the concept of a point with generic polar curve should consult \cite{ms} or just consider that there is a zero-dimensional stratum at that point.
\begin{theorem}[\cite{ms} Theorem 2.3]
\[
\tilde{b}_n(F) -\tilde{b}_{n-1} (F) 
= \sum_p \tilde{b}_n(F_p )+ \sum_q \left( \tilde{b}_n(F_q) - \tilde{b}_{n-1}(F_q) \right)
- \sum_k \stackrel{\circ}{\mu }_k \left( \chi (\Sigma_a^k) - \sum_{q\in \Sigma_a^k} 1 \right),
\]
where $F$ denotes the Milnor fibre $f_0$ at the origin, $F_x$ denotes the Milnor fibre of $f_a-f_a(x)$ at the point $x$, the
$p$'s are summed over all $p\in B_\epsilon \cap \Sigma f_a$ which are not contained in $V(f_a)$, the $q$'s are summed over those
$q$ in $B_\epsilon \cap \Sigma f_a $ which are contained in $V(f_a)$ and at which $f_a$ has generic polar curve, the $\{ \Sigma_a^k
: k=1,2,\dots \} $ are the irreducible components of $\Sigma V(f_a)$ in $B_\epsilon $ and $\stackrel{\circ}{\mu}_k$ is the generic transverse Milnor number of $\Sigma _a^k$, the $k$th component of $\Sigma (f_a)$.
\end{theorem}
\begin{proof}
By Lemmas~2.1 and 2.2 of \cite{ms} we have a quasi-Milnor fibration. 
The general fibre $\gf $ is connected and has non-trivial Betti numbers $b_n$ and $b_{n-1}$ by the Kato-Matsumoto Theorem in \cite{km}. By Theorem~2.3 of \cite{siersma-wark},  $\sf$ is connected and is homotopically equivalent to a wedge of spheres. (Note that in his statement Siersma has an erroneous extra $\partial $.) Hence, we have
\[
\chi (\gf ) - \chi (\sf ) = 
\widetilde{\chi } (\gf) - \widetilde{\chi }(\sf ) = 
(-1)^n\left( \tilde{b}_n(F) - \tilde{b}_{n-1}(F) - \mu (F_p) \right) .
\] 

We can use the statement of our main theorem, Theorem~\ref{maineulerthm}:
\[
\chi (\gf ) - \chi (\sf ) = \sum_{X\in \SS }  \widetilde{\chi } \left( \tmf _X \right) \chi \left( \cl{X}, \cl{X}\backslash X \right) .
\]
We can re-express the left-hand side of this equation using the previous equation. To re-express the right-hand side we need to stratify the special fibre.
We can stratify by taking as $0$-dimensional strata all the points that have generic polar curve. 
The $1$-dimensional strata are the (connected) components of $\Sigma ^k_a\backslash \{ q_1, \dots , q_s\}$.
The $n$-dimensional stratum comes from the complement of the union of $\Sigma ^k_a$ in $\sf $, since the map is reduced, we see that the contribution from the $n$-dimensional stratum to the right-hand side is zero.

We then have
\begin{eqnarray*}
& & \sum_{X\in \SS }  \widetilde{\chi } \left( \tmf _X \right) \chi \left( \cl{X}, \cl{X}\backslash X \right) \\
&=& \sum_q \widetilde{\chi } \left( \tmf _q \right) \chi (\{ q \} , \emptyset) 
 + \sum_k \widetilde{\chi } \left( \tmf _{\Sigma^k_a \backslash \{ q_i \}} \right)\chi \left(\Sigma ^k_a , \cup_{q\in \Sigma ^k_a} \{ q \} \right)   \\
&=& \sum_q \widetilde{\chi } (F_q) 
 + \sum_k  \widetilde{\chi} \left( \tmf _{\Sigma^k_a \backslash \{ q_i \}}  \right)\left( \chi \left(\Sigma ^k_a \right) - \chi \left( \cup_{q\in \Sigma ^k_a} \{ q \} \right)  \right)   \\
&=& \sum_q \widetilde{\chi } (F_q) 
+ \sum_k    (-1)^{n+1-2} \stackrel{\circ}{\mu} _k  \left( \chi \left(\Sigma ^k_a \right) - \sum_{q\in \Sigma ^k_a} 1  \right)   \\
&=& \sum_q (-1)^n \left( \widetilde{b}(F_q) - \widetilde{b}_{n-1} (F_q) \right) 
 + \sum_k    (-1)^n (-1) \stackrel{\circ}{\mu} _k  \left( \chi \left(\Sigma ^k_a \right) - \sum_{q\in \Sigma ^k_a} 1  \right) .
\end{eqnarray*}
From this and the expression for the left-hand side we get the theorem.
\end{proof}

We know turn our attention to composed singularities. These occur when we take the composition of two known singularities to produce a third.
In particular, suppose that $f:(\C ^n,0)\to (\C ^2,0)$, $n\geq 2$, is complex analytic map germ defining an isolated complete intersection singularity and that $g:(\C ^2,0)\to (\C ,0)$ defines an isolated curve singularity.
The composition $F:(\C ^n,0)\to (\C ,0)$ given by $F=g\circ f$ is a non-isolated hypersurface singularity since the singular locus is $f^{-1}(0)$.

These are known as {\em{generalized Zariski examples}} and have been studied in many papers. For us particular references are \cite{nemethi} p61, \cite{damon-highmult} p92 and \cite{ms} p770.

Denoting the components of $f$ by $(f_1,f_2)$, the deformation $F_a(x)=g(f_1(x)-a_1,f_2(x)-a_2)$ is equi-transversal for generic $a=(a_1,a_2)$. As in the previous theorem, we can set $h(x)=F_a(x)$ and restrict to a small ball for such a generic $a$. The critical locus of $F_a$ is $f^{-1}(a)$ plus a collection of isolated singularities. The special fibre of $h$ is a wedge of spheres, the number of which is equal to the sum of the Milnor numbers of the isolated singularities, see Theorem~2.3 of \cite{siersma-wark}, denote this number by $\mu (\sf )$. The singular set of $F_a$ is $f^{-1}(a)$ and the transversal Milnor fibre of this manifold is the Milnor fibre of $g$. Hence the stratification of the special fibre is $\SS = \left\{  F_a^{-1} (0) \backslash f^{-1}(a) ,f^{-1}(a)  \right\}$. The generic fibre of $h$ is the Milnor fibre of $F$, denoted $\mf _F$.

\begin{theorem}[See \cite{ms}]
\label{msthm}
In this situation we have 
\[
\chi  (\mf _F) = \chi \left( F_a^{-1}(0) \right)  - \mu (g) \chi \left( \mf _f \right) .
\]
\end{theorem}
Nem\'ethi gives a similar formula in \cite{nemethi}.
\begin{proof}
From Theorem~\ref{maineulerthm} we have 
\begin{eqnarray*}
\chi (\gf ) - \chi (\sf ) &=& \sum_{X\in \SS } \left( \chi (\tmf _X)-1 \right) \chi \left( \cl{X}, \cl{X}\backslash X \right) \\
\chi (\mf _F) &=& \chi (\sf ) +(1-\mu (g) -1 ) \chi \left( f^{-1}(a) , \emptyset \right) \\
 &=& \chi \left( F_a^{-1}(0) \right) - \mu (g) \chi \left( f^{-1}(a) \right) \\
 &=&  \chi \left( F_a^{-1}(0) \right) -  \mu (g) \chi \left( \mf _f \right) . \\
\end{eqnarray*}
\end{proof}
In \cite{damon-highmult} Corollary 10.4 Damon shows that
\[
\widetilde{\chi } \left( \mf _F \right) = \mu _I (f) + (-1)^n   \mu (g) \chi \left( \mf _f \right) 
\]
where $\mu _I (f) $ is the image Milnor number for $f$, defined in the next section. 

With this result and Theorem~\ref{msthm} we find that
\[
 \chi \left( F_a^{-1}(0) \right) = \mu _I(f) +\left( (-1)^n +1  \right) \mu (g) \chi \left( \mf _f \right) +1 
 \] 
 and in particular if $n$ is odd, then 
 \[
\widetilde{\chi} \left( F_a^{-1}(0) \right) = \mu _I(f) .
 \] 

\section{Disentanglements and multiple point spaces}
\label{sec:dis}
Much of this section can be developed for multi-germs but for clarity we shall only develop the theory for mono-germs.

Let $f:(\C ^n,0)\to (\C ^p,0)$ be a map-germ with an isolated instability at the origin. That is, for a representative of $f$, which we also denote by $f$, the multi-germ $f:(\C ^n ,\Sigma f \cap f^{-1}(y))\to (\C ^p,y)$ is stable for all $y\neq 0$. Here $\Sigma f$ is the critical set of $f$, i.e., the points where the differential of $f$ has less than maximal rank.

\begin{definition}
There exists a topological stabilization of $f$, that is a map $\widetilde{f}:U\to \C^ p$ where $U$ is an open neighbourhood of the origin in $\C ^n$ and $\widetilde{f}$ is topologically stable at every point.
The discriminant of $\widetilde{f}$ is called the {\em{disentanglement of $f$}} and is denoted $\dis (f)$.
\end{definition}
If $f$ is corank $1$ or in the nice dimensions, then the stabilization is smooth. 
The precise details for the construction of the disentanglement can be found in \cite{marar}.

Consider the case of $p\leq n+1$. Here the disentanglement is a hypersurface and for $n>1$ it has non-isolated singularities.
It is shown in \cite{dm} and \cite{vancyc} that it is homotopically equivalent to a wedge of spheres of dimension $p-1$. Note that in the case $p=n+1$, the discriminant is just the image of $\widetilde{f}$.

An obvious task is to compare the topology of the disentanglement of $f$ and the Milnor fibre of the hypersurface given by the discriminant of $f$. Theorem~\ref{maineulerthm} allows us to do this when we view the disentanglement as a special fibre and the Milnor fibre of the discriminant as the general fibre. That we can find $h$ from Section~\ref{sec:notation} in the case $p=n+1$ is effectively described in \cite{vancyc}, but see also \cite{siersma-wark}. The statement for the $n\geq p$ case is similar.

Theorem~\ref{maineulerthm} requires that we know the Euler characteristic of the transverse Milnor fibre. In our applications a stratum in the special fibre will arise from the discriminant of a stable map. When the stable map has more than one branch, then this Euler characteristic is zero as we see in the next lemma.

First we shall say what we mean by stratification by stable type. 
Let $G:(\C ^n,\{ x_1, \dots ,x_s\}  )\to (\C ^p,0)$ be a stable multi-germ. 
There exist open sets $U\subseteq \C ^n $ and $W \subseteq \C ^p$ such that 
$G^{-1}(W) = U $ and $G:U\to W$ is a representative of $G$.
We can partition the discriminant of $G$ by stable type. That is, $y_1$ and $y_2$ in $\C ^p$
have the same stable type if $G_1:(\C ^n , G_1^{-1}(y_1)\cap \Sigma G_1 ) \to (\C ^p,y_1)$ and
$G_2:(\C ^n , G_2^{-1}(y_2)\cap \Sigma G_2 ) \to (\C ^p,y_2)$ are $\AA $-equivalent. 
These sets are complex analytic manifolds. 

If $G$ is corank $1$ or in the nice dimensions, then this partition is the canonical Whitney stratification. See Section~6 of \cite{polar} or Section~2.5 of \cite{dupwall}.
\begin{lemma}
\label{sakamoto}
Let $G:(\C ^n ,\{ x_1, \dots ,x_s\} )\to (\C ^p,0)$ be a stable multi-germ with $n\geq p-1$. If we can Whitney stratify the discriminant by stable type, then for any stratum $X$, we have $\chi (\tmf _X) =0$.
\end{lemma}
\begin{proof}
By Mather's classic result \cite{mather-stable} each branch is stable and the analytic manifolds for these branches intersect transversely. Without loss of generality we can assume that the stratum is zero-dimensional. 

If $g_i$ is a function defining the discriminant in $\C ^p$ for the branch $i$, then the discriminant of $G$ is defined by $g_1(z_1)g_2(z_2)\dots g_s(z_s)$ where $z_i$ is a set of coordinates in $\C ^p$ such that the sets $z_i$ and $z_j$ share no common coordinates for $i\neq j$.

Theorem~2 of  \cite{sakamoto} states that the Milnor fibre of a product of two functions is homotopically equivalent to the total space of a fibre bundle over the circle. Hence, the Euler characteristic of the Milnor fibre of a product is zero. Induction on the number of branches produces the result we require.
\end{proof}
Noting that a stratum will either arise from (the trivial unfolding of) a mono-germ or a multi-germ, 
from the lemma we can deduce the following.
\begin{theorem}
\label{dis-formula}
Suppose that $f:(\C ^n,0)\to (\C ^p,0)$, $n\geq p$, has an isolated instability and $f$ is corank $1$ or is in the nice dimensions. Let $MF_g$ denote the Milnor fibre of the function $g$ that defines the image of $f$.  Then we have
\begin{eqnarray*}
\chi (\mf _g ) - \chi (\dis (f)  ) &=& 
\sum_{{\text{\begin{tabular}{c}$X$ arises from \\ a mono-germ\end{tabular}}}} \widetilde{\chi} \left(  \tmf _X \right) \chi \left( \cl{X}, \cl{X}\backslash X \right) \\
& & \qquad -\sum_{{\text{\begin{tabular}{c}$X$ arises from \\ a non-mono-germ\end{tabular}}} }  \chi \left( \cl{X}, \cl{X}\backslash X \right) .
\end{eqnarray*}
\end{theorem}
\begin{proof}
By a straightforward generalization of Proposition~4.4 of \cite{dm} to the case of $p=n+1$ and Theorem~2.3 of \cite{siersma-wark} (see also Section~4 of \cite{damon-highmult}) we can see that we have a quasi-Milnor fibration where the special fibre is the disentanglement of $f$ and the general fibre is the Milnor fibre of the function defining the image of $f$.

The result the follows from Theorem~\ref{maineulerthm} and Lemma~\ref{sakamoto}.
\end{proof}

Another important result we shall need concerns the multiple point spaces of the disentanglement map in the case of corank $1$ maps with $p=n+1$. For this we shall require the main result of \cite{mm}. That is, the multiple point spaces of the map $f$ can be defined using Vandermonde determinants and these are isolated complete intersection singularities. The corresponding multiple point spaces for the disentanglement map are then the Milnor fibres of these singularities.

The $k$th multiple point space of $f$, denoted $D^k(f)$, is defined using Vandermonde determinants as in \cite{mm}. For maps with isolated instabilities these are, in general, isolated complete intersection singularities and so have a Milnor number.
For stable maps, the multiple point spaces are non-singular, and hence the multiple point spaces for the disentanglement map are the Milnor fibres of the multiple point spaces of $f$.

Furthermore, $D^k(f)\subset \C ^{n+k+1}$ and the group of permutations on $k$ objects, denoted $S_k$, acts on $\C ^{n+k+1}$ by permutation of $k$ of the coordinates. With this we can define certain important subspaces as in Section 3 of \cite{mm} which we now describe.

Let $\PP = (k_1, \dots , k_1, k_2, \dots , k_2, \dots , k_r, \dots , k_r)$ be a partition of $k$ with $1\leq k_1 <k_2 < \dots < k_r$ and $k_i$ appearing $\alpha _i$ times (so $\sum \alpha _i k_i = k$). For every $\sigma \in S_k$ there is a partition $\PP _\sigma $, that is, in its cycle decomposition there are $\alpha _i$ cycles of length $k_i$. For any partition we can associate any one of a number of elements of $S_k$.

For a map $f$ we define $D^k(f,\PP _\sigma )$ to be the restriction of $D^k(f)$ to the fixed point set of $\sigma $. This is well-defined for a partition, i.e., if $\PP _{\sigma } = \PP _{\sigma '}$, then 
$D^k(f,\PP _{\sigma })$ is equivalent as an $S_k$-invariant isolated complete intersection
singularity to $D^k(f,\PP _{\sigma '})$.

As $D^k(f,\PP )$ is an isolated complete intersection singularity it has a Milnor fibre. We denote this by  $\mu \left( D^k(f,\PP )\right)$, where we take $\mu \left( D^k(f,\PP) \right) =0$ if $D^k(f,\PP )=\emptyset $.

If we can Whitney stratify the special fibre by stable type, then $\chi \left( \cl{X}, \cl{X}\backslash X \right)$ can be calculated using the $\mu \left( D^k(f ,\PP )\right)$ as we can see in the following theorem and corollary. 
\begin{theorem}
\label{firstlinsum}
Suppose that $f:(\C ^n,0)\to (\C ^{n+1},0)$ is a corank $1$ map with an isolated instability at $0$. Then, the disentanglement of $f$ can be Whitney stratified by stable type and for any stratum $X$ we have
\[
\chi \left( \overline{X} \right) = \chi \left( \cl{X} \right)= \sum _{k,\PP } c_{k,\PP } \mu \left( D^k(f ,\PP )\right) 
\]
where $c_{k,\PP }$ is a (possibly zero) rational number.
\end{theorem}
\begin{proof}
We stratify the disentanglement by stable type. Let $X$ be a stratum. Since the disentanglement map is a corank $1$ stable map $X$ will correspond to the top stratum in some $D^k(\widetilde{f},\PP )$ which by the Marar-Mond description is the Milnor fibre of the isolated complete intersection singularity $\widetilde{D}^k(\widetilde{f},\PP )$.

If $X$ is zero dimensional, then the Milnor number of $D^k(f,\PP )$ plus $1$ and divided by a suitable factor depending on $\PP $ will calculate $\chi \left( \overline{X} \right) = \chi (X)$.

For higher dimensional strata we use the double induction technique of Theorem~2.8 of \cite{gomo}, and also used in Section~4.3 of \cite{bojo} and Section~2 of \cite{calcgen}. That is, the closure of the stratum is the image of the restriction of the disentanglement map to some set $S$ say. The multiple point spaces of $f|S$ are the same as those of $f$ except those that are the image of some map. This map has the same properties of $f$ that allow a double induction. The Euler characteristic of the alternating homology of the spaces can be calculated using the material in Section~2.2 of \cite{hk}. (See also \cite{excellent}.)
\end{proof}

\begin{corollary}
\label{linsum}
Suppose that $f:(\C ^n,0)\to (\C ^{n+1},0)$ is a corank $1$ map with an isolated instability. Then, the disentanglement can be Whitney stratified by stable type and for any stratum $X$ we have
\[
\chi \left( \overline{X} , \cup_{Y<X} \overline{Y}  \right) = \chi \left( \cl{X} , \cl{X} \backslash X \right)= \sum _{k,\PP } b_{k,\PP } \mu \left( D^k(f ,\PP )\right)
\]
where $b_{k,\PP }$ is a (possibly zero) rational number.
\end{corollary}
\begin{proof}
The theorem above shows that $\chi \left( \overline{X} \right)= \sum _{k,\PP } c_{k,\PP } \mu \left( D^k(f ,\PP ) \right)$.
We need to show that $\cup_{Y<X} \overline{Y}$ has the same property. But $\cup_{Y<X} \overline{Y}=\cup_{i=1}^r \overline{Y_i}$ for some $r$ and some $Y_i$ where $\overline{Y_i}\not\subseteq \overline{Y_j}$  for all $i,j$ with $i\neq j$. Hence, we must show that
\begin{eqnarray}
\label{eqn:induction}
\chi \left( \cup_{i=1}^r \overline{Y_i}\right) = \sum _{k,\PP } d_{k,\PP } \mu \left( D^k(f ,\PP )\right) 
\end{eqnarray} 
for some constants $d_{k,\PP }$.
To do this we use double induction on $r$ and the dimension of $\cup_{i=1}^r \overline{Y_i}$. 

For zero-dimensional $Y_i$ it is easy to show that the statement is true for all $r$ as $Y_i$ corresponds to some zero-dimensional multiple point space of the disentanglement map. 

Thus suppose that (\ref{eqn:induction}) hold for all $\cup_{i=1}^r \overline{Y_i}$ of dimension less than $d$ and for all $r$.
Now consider a union where $\cup_{i=1}^r \overline{Y_i}$ has dimension $d$. 

The statement (\ref{eqn:induction}) holds when $r=1$ by the theorem. Thus suppose that the statement is true for $r$ and consider the case of $r+1$.
Using the standard inclusion-exclusion formula, $\chi (A\cup B)=\chi (A) + \chi (B) - \chi (A\cap B)$, we have
\begin{eqnarray*}
\chi \left( \cup_{i=1}^{r+1} \overline{Y_i} \right) &=& \chi \left( \cup_{i=1}^r \overline{Y_i} \right)
+ \chi \left(  \overline{Y_{r+1}} \right) - \chi \left( \left( \cup_{i=1}^r \overline{Y_i} \right) \cap \overline{Y_{r+1}} \right) \\
&=& \chi \left( \cup_{i=1}^r \overline{Y_i} \right)
+ \chi \left(  \overline{Y_{r+1}} \right) - \chi \left(  \cup_{i=1}^r \left( \overline{Y_i}  \cap \overline{Y_{r+1}} \right) \right) \\
&=& \chi \left( \cup_{i=1}^r \overline{Y_i} \right)
+ \chi \left(  \overline{Y_{r+1}} \right) - \chi \left(  \cup_{i=1}^q  \overline{Z_i}  \right) ,
\end{eqnarray*}
for some $q$ and $q$ strata $Z_i$, $1\leq i \leq q$. Note that $Z_i$ is not just $Y_i\cap Y_{r+1}$  as this intersection may be empty. However, since the disentanglement has corank $1$ singularities and has a finite number of strata, then there exists a finite number of $Z_i$ and each has dimension less than $d$. 

Therefore, on the right-hand side of the last equation the first term can be written as a sum of Milnor numbers by the inductive hypothesis on $r$, the second has the same representation by the theorem, and the third terms has the same property by the inductive hypothesis on the dimension of the union.

Hence, by induction, statement (\ref{eqn:induction}) holds for all $r$ and all dimensions.
\end{proof}

\section{Maps from $3$-space to $4$-space}
\label{sec:3to4}
To exemplify Corollary~\ref{linsum} we could use the results from the first half of \cite{diswe} which dealt with maps 
of the form $f:(\C ^2,0)\to (\C ^3,0)$. Instead we shall go to the next case, that of $f:(\C ^3,0)\to (\C ^4,0)$, as we shall need it when we generalize the results from the second half of \cite{diswe}.

First, we note that each singular point of the disentanglement is the image of a stable map. The precise list of possible singularities is the following.
\begin{itemize}
\item $A_{1}$: The point in the disentanglement is non-singular.
\item $A_{1,1}$: A transverse crossing of two $3$-planes, i.e., an ordinary double point.
\item $A_{1,1,1}$: A transverse crossing of three $3$-planes, i.e., an ordinary triple point.
\item $A_{1,1,1,1}$: A transverse crossing of four $3$-planes, i.e., an ordinary quadruple point.
\item $A_{2}$: The trivial unfolding of a Whitney cross-cap.
\item $A_{1,2}$: The transverse intersection of a $3$-plane and $A_2$.
\end{itemize}
We have the following adjacencies:

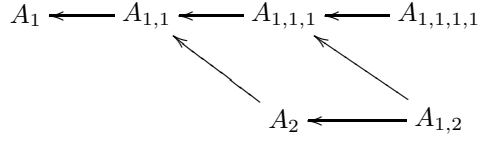
\begin{equation*}
\xymatrix{
A_1 &  \ar[l] A_{1,1}  & A_{1,1,1} \ar[l]  & A_{1,1,1,1} \ar[l] \\
&  &  A_2 \ar[lu]  & A_{1,2} \ar[lu]\ar[l]   \\
}
\end{equation*}

We shall let $S_X$ denote the set of points in the disentanglement associated to the singularity type $A_X$.
We have $\dim (S_1)=3$, $\dim S_{1,1}=2$, $\dim S_{1,1,1}=\dim S_{2}=1$ and $\dim S_{1,1,1,1}=\dim S_{1,2}=0$.
These sets are manifolds and form the canonical Whitney stratification of the disentanglement, see Section~6 of \cite{polar}.

Except for $A_2$ the maps corresponding to these singular sets are multi-germs and hence their transverse Milnor fibre will have Euler characteristic equal to zero by Lemma~\ref{sakamoto}. For $A_2$ the transverse Milnor fibre is the Milnor fibre of the function defining the image of the Whitney cross-cap. This has homotopy type of the $2$-sphere (see \cite{mass-le-hyp} page 4) and hence the transverse Milnor fibre for $S_2$ has Euler characteristic equal to $2$.

For any stratum $S$ we can calculate the Euler characteristic of the pair $(\overline{S}, \overline{S} \backslash S)$ using the material on $k$-disentanglements from \cite{calcgen}. 
The {\defn{$k$th disentanglement of $f$}} is the closure of the
set of points in the image of $\widetilde{f}$, the disentanglement map, that have $k$ or more preimages. We denote this by $\dis _k(f)$.
It is shown in \cite{bojo} says that if $f:(\C ^n,0)\to (\C ^{n+1},0)$ is
corank $1$ and finitely $\AA $-determined, then the $k$th disentanglement is
empty or is a bouquet of spheres of dimension $n-k+1$. The number of spheres in the bouquet is 
denoted by $\mu _{I_k}$. We shall write $\mu _{I_1}(f)$ simply as  $\mu _I(f)$.

We can relate the $k$th disentanglement to the sets $S_X$ and their closures. Furthermore, Theorem~2.8 in \cite{calcgen} shows how to calculate $\mu _{I_k}$ for a map $f:(\C ^3,0)\to (\C ^4,0)$ with an isolated instability in terms of the Milnor numbers of the multiple point spaces (which are isolated complete intersection singularities). Since $\dis _k(f)$ is a bouquet of spheres of a certain dimension, then we can calculate its Euler characteristic.

An important point to note is that in Theorem~2.8 of \cite{calcgen} certain terms in the formulae depend on the $k$th multiple point being non-empty. For example, in (ii) of Theorem~2.8 of \cite{calcgen} the $-1/3$ term is included only if $D^3$ is non-empty. Similarly in (iii). To keep track of this notion we shall make the following definition: 
\[
\notemp _k (f) = \left\{
\begin{array}{ll}
0, & {\text{ if }} \widetilde{D} ^k( \widetilde{f}) =\emptyset , \\
1, & {\text{ if }} \widetilde{D} ^k( \widetilde{f}) \neq \emptyset . 
\end{array}
\right.
\]

\begin{proposition}
\label{prop:strata}
For a map $f:(\C ^3,0)\to (\C ^4,0)$, and denoting 
by $Q$ the number of quadruple points appearing in the disentanglement $\widetilde{f}$, we have
\begin{eqnarray*}
\chi \left( \overline{S_1} \right) &=& 1- \left( \frac{1}{2} ( \mu (D^2) +\mu (D^2|H) )  +\frac{1}{6} ( \mu (D^3) +3\mu (D^3|H_1) +2\notemp _3 ) +Q \right) , \\
\chi \left( \overline{S_{1,1}} \right) &=&\notemp _2+ \frac{1}{2} \left( \mu (D^2) -\mu (D^2|H) \right)  +\frac{1}{3} \left(\mu(D^3) -\notemp _3\right) +3Q,\\
\chi \left( \overline{S_{1,1,1}} \right) &=&\notemp _3 - \left( \frac{1}{6} \left( \mu (D^3)-3\mu(D^3|H_1) +2\notemp _3 \right) +3Q \right) ,\\
\chi \left( \overline{S_2} \right) &=& \chi \left( D^2|H \right) = \notemp _2- \mu \left( D^2|H \right), \\
\chi \left( \overline{S_{1,1,1,1}} \right) &=& \chi \left( D^4\right) = Q ,\\
\chi \left( \overline{S_{1,2}} \right) &=& \chi \left( D^3|H_1\right) = \mu \left( D^3|H_1 \right) +\notemp _3 . 
\end{eqnarray*}
\end{proposition}
\begin{proof}
Obviously $\dis (f)=\dis _1(f)=\overline{S_1}$. The number $\mu _I=\mu _{I_1} $ counts the number of spheres in the homotopy type of this space and so from Theorem~2.8 of \cite{calcgen}, we have
\[
\chi \left( \overline{S_1} \right) =
1- \left( 
\frac{1}{2} ( \mu (D^2) +\mu (D^2|H) ) 
+\frac{1}{6} ( \mu (D^3) +3\mu (D^3|H_1) +2\notemp _3 ) +Q \right) .
\]
Note that $Q=\chi \left( S_{1,1,1,1}\right) =\mu \left( D^4(f) \right) +\notemp _4 $.

The set $\dis _2(f)$ is homotopically equivalent to a bouquet of $\mu _{I_2}$ spheres and coincides, as a set, with $\overline{S_{1,1}}$. From the proof of Theorem~2.8 of \cite{calcgen} we find 
\[
\chi \left( \overline{S_{1,1}} \right) =
\notemp _2 + \frac{1}{2} ( \mu (D^2) -\mu (D^2|H) )  +\frac{1}{3} (\mu(D^3) -\notemp _3 ) +3Q.
\]

Similarly $\dis _3=\overline{S_{1,1,1}}$ and so again from Theorem~2.8 of \cite{calcgen}, we find
\[
\chi \left( \overline{S_{1,1,1}} \right) =
\notemp _3 - \left( \frac{1}{6} \left( \mu (D^3)-3\mu(D^3|H_1) +2\notemp _3 \right) +3Q \right) .
\]

A point in the curve $S_2$ corresponds to a Whitney cross-cap point. These correspond to points of $D^2|H$. The correspondence is in fact a bijection. Hence, as $\overline{S_2}$ is homotopically equivalent to a wedge of circles, we see that 
\[
\chi \left( \overline{S_2} \right) = \chi \left( D^2|H \right) = \notemp _2 - \mu \left( D^2|H \right) .
\]
The spaces $S_{1,1,1,1}$ and $S_{1,2}$ are zero-dimensional and so 
\[
\chi \left( \overline{S_{1,1,1,1}} \right) = \chi \left( D^4\right) = Q 
\]
and
\[
\chi \left( \overline{S_{1,2}} \right) = \chi \left( D^3|H_1\right) = \mu \left( D^3|H_1 \right) +\notemp _3 .  
\]
\end{proof}

We can now use the adjacency diagram for the singularities to calculate  $\chi \left( \overline{S_X}, \overline{S_X} \backslash S_X\right) $ for every stratum $S_X$. 
\begin{proposition}
\label{prop:relstrata}
We have
\begin{eqnarray*}
\chi \left( \overline{S_{1,1}}, \overline{S_{1,1}} \backslash S_{1,1} \right) 
&=&  \chi \left( \overline{S_{1,1}}\right) - \chi \left(  \overline{S_{1,1,1}} \right) - \chi \left(  \overline{S_{2}}  \right) + \chi \left(  S_{1,2}   \right) ,\\
\chi \left( \overline{S_{1,1,1}}, \overline{S_{1,1,1}} \backslash S_{1,1,1} \right) 
&=&\chi \left( \overline{S_{1,1,1}} \right) - \chi \left( S_{1,1,1,1} \right)  - \chi \left( S_{1,2}  \right) , \\
\chi \left( \overline{S_{2}}, \overline{S_{2}} \backslash S_{2} \right) 
&=& \chi \left( \overline{S_{2}}\right) - \chi \left( S_{1,2} \right).
\end{eqnarray*}
\end{proposition}
\begin{proof}
For $S_{1,1}$ we have 
\begin{eqnarray*}
\chi \left( \overline{S_{1,1}}, \overline{S_{1,1}} \backslash S_{1,1} \right) 
&=& \chi \left( \overline{S_{1,1}}, \overline{S_{1,1,1}} \cup  \overline{S_{2}}  \right) \\
&=& \chi \left( \overline{S_{1,1}}\right) - \chi \left(  \overline{S_{1,1,1}} \cup  \overline{S_{2}}  \right) \\
&=& \chi \left( \overline{S_{1,1}}\right) - \chi \left(  \overline{S_{1,1,1}} \right) - \chi \left(  \overline{S_{2}}  \right) + \chi \left(  \overline{S_{1,1,1}} \cap  \overline{S_{2}}  \right) \\
&=& \chi \left( \overline{S_{1,1}}\right) - \chi \left(  \overline{S_{1,1,1}} \right) - \chi \left(  \overline{S_{2}}  \right) + \chi \left(  S_{1,2}   \right) .
\end{eqnarray*}
Therefore, we can calculate the Euler characteristic of this pair using the Milnor numbers of the multiple point spaces.

For the stratum $S_{1,1,1}$ we see that 
\begin{eqnarray*}
\chi \left( \overline{S_{1,1,1}}, \overline{S_{1,1,1}} \backslash S_{1,1,1} \right) 
&=& \chi \left( \overline{S_{1,1,1}}, S_{1,1,1,1} \cup S_{1,2}  \right) \\
&=& \chi \left( \overline{S_{1,1,1}} \right) - \chi \left( S_{1,1,1,1} \cup  S_{1,2}  \right) \\
&=&  \chi \left( \overline{S_{1,1,1}} \right) - \chi \left( S_{1,1,1,1} \right)  - \chi \left( S_{1,2}  \right) .
\end{eqnarray*}
So, again we can calculate the Euler characteristic in terms of Milnor numbers of multiple point spaces.

For $S_2$ we have 
\begin{eqnarray*}
\chi \left( \overline{S_{2}}, \overline{S_{2}} \backslash S_{2} \right) 
&=& \chi \left( \overline{S_{2}}, S_{1,2} \right) \\
&=& \chi \left( \overline{S_{2}}\right) - \chi \left( S_{1,2} \right) .
\end{eqnarray*}
\end{proof}

Combining these two preceding propositions with Corollary~\ref{dis-formula} we produce the following.
\begin{theorem}
\label{c3toc4thm}
Suppose that $f:(\C ^3,0)\to (\C ^4,0)$ is a corank $1$ germ with an isolated instability at $0$.
Let $\mf _g$ denote the Milnor fibre of the function $g$ defining the image of $f$. Then,
\[
\chi (\mf _g) - \chi (\dis (f)) = 
\notemp _2 -\left( 
\frac{1}{2} \left( \mu (D^2) +3\mu (D^2|H) \right) 
+\frac{1}{3} \left( \mu (D^3) +6\mu (D^3|H_1) +5 \notemp _3 \right) +3Q
\right) .
\]
\end{theorem}
\begin{proof}
From Corollary~\ref{dis-formula} we know that the left-hand side of the equation is equal to 
\[
\sum_{X\in \SS } \left( \chi (\tmf _X)-1 \right) \chi \left( \cl{X}, \cl{X}\backslash X \right) .
\]
where $\SS $ is the stratification by stable type.
Since $\chi (\tmf _S )$ is zero for all strata $S$ except for $S_2$, when it is $2$, we have that
the left hand side is equal to

\[
2\chi \left( \overline{S_2}, \overline{S_2} \backslash S_2\right) 
-\sum_{S_X\in \SS {\text{ and }} S_X\neq S_2 }  \chi \left( \overline{S_X}, \overline{S_X} \backslash S_X\right)  .
\]
Using the descriptions for 
$ \chi \left( \overline{S_X}, \overline{S_X} \backslash S_X\right) $ in Proposition~\ref{prop:relstrata} we deduce the result.
\end{proof}

Since Proposition~\ref{prop:strata} allows us to express $\chi \left( \dis (f) \right)$, which is  equal to  $\chi (\overline{S_1})$, as a sum of Milnor numbers of multiple point spaces we deduce the following.
\begin{corollary}
\label{c3toc4cor}
Suppose that $f:(\C ^3,0)\to (\C ^4,0)$ is a corank $1$ germ with an isolated instability at $0$. Then,
\[
\chi (\mf _g) = 1+ \notemp _2 -\left( 
 \mu (D^2) +2\mu (D^2|H) ) 
+\frac{1}{2} \left( \mu (D^3) +5\mu (D^3|H_1) \right) +2\notemp _3 +4Q
\right) .
\]
\end{corollary}

\begin{remarks}
\begin{enumerate}

\item An interesting consequence of the equation is that $\mu (D^3)$ and $\mu (D^3|H_1)$ must have the same parity.

\item In \cite{hk} the Milnor numbers of the multiple point spaces were calculated for some families and for all  simple corank $1$ germs from $\C ^3$ to $\C ^4$. We can use these to calculate the Euler characteristic of the Milnor fibre for the image of the map. This is given in Table~\ref{table-for-mf}.
\end{enumerate}
\end{remarks}
\begin{table}
\caption{{\label{table-for-mf}}Euler characteristic of Milnor fibres}
\begin{tabular}{|l|c|c|c|}
\hline
Singularity & Name & $\mu _I$ & $-\chi (\mf )$ \\
\hline
$(x,y,z^2,z(z^2\pm x^2\pm y^{k+1}))$, $k\geq 1$ & $A_k$ & $k$& $3k-2$  \\
$(x,y,z^2,z(z^2+x^2y\pm y^{k-1}))$, $k\geq 4$ &$D_k$ &$k$ &$3k-2$  \\
$(x,y,z^2,z(z^2+x^3\pm y^4))$ & $E_6$ &$6$ & $16$ \\
$(x,y,z^2,z(z^2+x^3+xy^3))$ & $E_7$ &$7$ & $19$   \\
$(x,y,z^2,z(z^2+x^3+y^5))$ & $E_8$ &$8$ & $22$ \\
$(x,y,z^2,z(x^2\pm y^2\pm z^{2k}))$, $k\geq 2$ &$B_k$ &$k$&$2k-1$  \\
$(x,y,z^2,z(x^2+yz^2\pm y^k))$, $k\geq 3$ & $C_k$  &$k$ &$3k-3$  \\
$(x,y,z^2,z(x^2+y^3\pm z^4))$ & $F_4$& $4$ & $8$  \\
\hline
$(x,y,yz+z^4,xz+z^3)$ &$P_1$ & 1& 3 \\
$(x,y,yz+z^5,xz+z^3)$ &$P_2$ & 2& 7 \\
$(x,y,yz+z^6\pm z^{3k+2},xz+z^3)$, $k\geq 2$ &$P_3^k$ &$k+2$ &$3k+8$   \\
$(z,y,yz+z^7+z^8,xz+z^3)$ &$P_4^1$ &5 &18  \\
$(x,y,yz+z^7,xz+z^3)$ &$P_4$ &5 &18  \\
$(x,y,yz+z^{k+3},xz+z^3)$, $k\neq 3j$ &$P_k$ & $\frac{1}{6}(k+1)(k+2)$ & $\frac{1}{2}(k^2+5k)$  \\
\hline
$(x,y,xz+yz^2,z^3\pm y^kz)$ &$Q_k$, $k\geq 2$ &$k$ &$3k$   \\
$(x,y,xz+z^3,yz^2+z^4+z^{2k-1})$, $k\geq 3$ &$R_k$ &$k+1$ &$2k+6$  \\
$(x,y,xz+y^2z^2\pm z^{3j+2},z^3\pm y^kz)$, $j\geq 1$, $k\geq 2$ & $S_{j,k}$ &$k+j+1$ &$3(k+j)+5$  \\
\hline
$(x,y,yz+xz^3\pm z^5+az^7,xz+z^4+bz^6)$ & I &6 &19  \\
$(x,y,yz+xz^3+az^6+z^7+bz^8+cz^9,xz+z^4)$ & II & 9& 30 \\
$(x,y,yz+z^5+z^6+az^7,xz+z^4)$ & III& 6 & 19 \\
$(x,y,yz+z^5+az^7,xz+z^4\pm z^6)$ &IV &6 &19  \\
$(x,y,xz+z^5+ay^3z^2+y^4z^2,z^3\pm y^2z)$ &V &6 &22  \\
$(x,y,xz+z^3,yz^2+z^5+z^6+az^7)$ &VI &6 &19  \\
$(x,y,xz+z^3,y^2z+xz^2+az^4\pm z^5)$ &VII & 6& 19 \\
$(x,y,xz+z^4+az^6+bz^7,yz^2+z^4+z^5)$ & VIII & 8& 24  \\
\hline
\end{tabular}
\end{table}

\section{An example in the equidimensional case}
\label{sec:equidim}

We shall now look at a very simple example of a family of maps from $\C ^n$ to $\C ^n$ which Terry Gaffney suggested to me as a useful source of examples.

Suppose now that $f:(\C ^n,0)\to (\C ^n,0)$ is a map of the form 
\[
f(x_1,\dots , x_{n-1},y) = \left( x_1,\dots , x_{n-1}, y^3+ \varphi (x_1,\dots , x_{n-1})y  \right)
\]
where $\varphi :(\C ^{n-1},0)\to (\C ,0)$  
defines an isolated hypersurface singularity at $0$. This latter condition implies that $f$ has an isolated instability at the origin.

Let $(X_1,\dots ,X_{n-1} ,Y)$ denote coordinates on the codomain of $f$. 
It is easy to show that discriminant of $f$ (i.e., the image of the critical set) is given as the zero-set of the map $H:(\C^n,0) \to (\C,0)$ defined by 
\[
H(X_1, \dots ,X_{n-1}, Y)= 4 \left( \varphi (X_1,\dots ,X_{n-1} ,Y) \right) ^3 +27Y^2 .
\]
By Theorem~1 of \cite{sakamoto} the Milnor fibre of $H$, denoted $\mf _H$, is homotopically equivalent to the suspension of the Milnor fibre of $\varphi ^3$. Hence, as $\chi \left( {\text{Susp}} (A) \right) = 2 - \chi (A) $, we have 
\begin{eqnarray*}
\chi \left( \mf _H \right) &=& 2 - \chi \left( \mf _{\varphi ^3} \right) \\ 
&=& 2 - 3 \chi \left( \mf _{\varphi } \right) \\ 
&=& 2 - 3 \left( 1+ (-1)^{n-2} \mu (\varphi ) \right) \\ 
&=& -1 + 3  (-1)^{n-1} \mu (\varphi )  . 
\end{eqnarray*}
We shall now verify the formula in Corollary~\ref{dis-formula} with this.

We can unfold $f$ with $F:(\C ^{n+1},0)\to (\C ^{n+1},0)$ given by 
\[
F(s, x_1,\dots , x_{n-1},y) = \left( s, x_1,\dots , x_{n-1}, y^3+ \left( \varphi (x_1,\dots , x_{n-1})+s\right) y  \right) 
\]
and we can define $f_s:B_\epsilon (0)\to \C ^n$ by 
\[
f_s(x)=\left( x_1,\dots , x_{n-1}, y^3+ \left( \varphi (x_1,\dots , x_{n-1})+s\right) y  \right).
\]

The critical set $\Sigma f_s$ is given by $3y^2+\varphi (x_1,\dots ,x_{n-1})+s$ and so, for small $s\neq 0$, $\Sigma f_s$ is the Milnor fibre of $\Sigma f_0$. Therefore, by Theorem~1 of \cite{sakamoto} it is homotopically equivalent to the suspension of the Milnor fibre of $\varphi $.
Also, $f_s$ is stable - its only singularities are cusps - so the disentanglement of $f$ is $\dis (f)=f_s\left( \Sigma f_s\right) $. Since the map $f_s|\Sigma f_s$ is injective the disentanglement is therefore homotopically equivalent to a bouquet of $\mu (\varphi )$ spheres of real dimension $n-1$.

In the quasi-Milnor fibration given by the disentanglement map we see that the disentanglement has two strata - the non-singular set and a cuspidal edge. This cuspidal edge, denoted $\CC $ is given as the zero-set of the map $(X_1,\dots ,X_{n-1},Y)\mapsto \left( \varphi (X_1,\dots ,X_{n-1})   ,Y\right) $. Thus, it is homotopically equivalent to the Milnor fibre of $\varphi $. 
Along the singular set, denoted $\CC $, the transverse Milnor fibre is the Milnor fibre of a cusp (i.e., the curve $x^3=y^2$), hence $\mu (\tmf _\CC )=2$.

In this quasi-Milnor fibration, the special fibre is the disentanglement of $f$ so $\chi (\dis(f)) = 1+(-1)^{n-1} \mu _(\varphi )$ and the general fibre is the Milnor fibre of the function defining the discriminant of $f$, i.e., the Milnor fibre of $H$.

From Corollary~\ref{dis-formula} we have
\begin{eqnarray*}
\chi (\mf _H)=\chi (\gf ) &=& \chi (\sf ) + \left( \chi (\tmf _\CC )-1  \right) \chi (\CC ,\emptyset ) \\
&=& \chi \left(\dis (f) \right) + \left( 1-2 -1 \right) \chi (\CC ) \\
&=& 1+ (-1)^{n-1} \mu \left( \varphi \right) -2 \left( 1+ (-1)^{n-2} \mu (\varphi ) \right) \\  
&=& -1 + 3 (-1)^{n-1} \mu (\varphi ) .
\end{eqnarray*}
Hence, the two calculations of the Euler characteristic of the Milnor fibre of the function defining the discriminant of $f$ coincide.

\section{Applications to equisingularity}
\label{sec:equising}
In this section we come to the main application of the Euler characteristic result. For this we need some preliminaries on equisingularity of mappings.
The idea is that we unfold a map $f:(\C ^n,0)\to (\C ^p,0)$ by a one-parameter unfolding and attempt to Whitney stratify this unfolding so that the parameter axis in the source and target are strata.
Most of the following can be done for multi-germs but we shall restrict to mono-germs for ease of exposition.

Throughout this section we shall assume that $f:(\C ^n , 0 ) \to (\C ^{n+1} ,0)$ is corank $1$ map with an isolated instability at the origin in $(\C ^{n+1},0)$.

Let $F:U\to W $ be a one-parameter unfolding of $f$. That is, $U\subseteq \C ^n \times \C $ and $W\subseteq \C ^{n+1}\times \C$ are open sets 
and $F(x,t)=(\overline{F}(x,t),t)$ so that $\overline{F}(x,t)|\left( U\cap \left( \C^n \times \{ 0 \} \right) \right) $ is a representative of $f$.

We say that $F$ is {\defn{origin-preserving}} if $\overline{F}(0,t)=0$ for all $t\in \C $.
We shall assume throughout that $F$ is origin-preserving.

Let $U_t=U\cap \left( \C ^n \times \{ t \}\right)$ and $W_t$ be defined similarly. Define $f'_t : U_t \to W_t$ by $f'_t(x)=F(x,t)=(\overline{F}(x,t),t)$ and let $f_t:(\C ^n , 0 ) \to (\C ^{n+1} ,0)$ be the germ at $0$ defined by $f'$.

Note that $F^{-1}(0,t)=(f'_t)^{-1}(0)$ contains $0\in \C ^n$ as $F$ is origin-preserving but it may in fact contain other points. This will be crucial in part (ii) of Definition~\ref{defn:excellent}.

\begin{definition}
Let $S=U\cap \left( \{ 0 \} \times \C \right )$ and $T=W\cap \left( \{ 0 \} \times \C \right )$ be the {\defn{parameter axes}} in source and target respectively.
\end{definition}

The aim is to stratify $F$ so that the parameter axes are strata, or more accurately, some neighbourhood of $0$ in these axes are strata.

\begin{definition}
We say that {\defn{$F$ is Whitney equisingular}} if there is a stratification of $F$ so that $F$ is a Thom $A_F$ map and $S$ and $T$ are Whitney strata.

We say that {\defn{the image of $F$ is Whitney equisingular}} if the image of $F$ can be given a Whitney stratification such that $T$ is a stratum.  
\end{definition}

Gaffney identified that one of the significant problems with stratifying the unfolding is that one-dimensional strata other than $S$ and $T$ may appear and these should be avoided. He proposed in \cite{polar} two definitions to identify unfoldings without these special strata:
\begin{definition}
\label{defn:excellent}
Suppose that $f:(\C ^n , 0 ) \to (\C ^{n+1} ,0)$ is corank $1$ and has an isolated instability at the origin in $(\C ^{n+1},0)$.
Suppose that $F$ is an origin-preserving one-parameter unfolding with a representative $F:U\to W$ which is proper and finite-to-one, and $F^{-1}(0) \cap U = \{  0 \} $. 
 
We call $F$ a {\em{good unfolding}} if all the following hold.
\begin{enumerate}
\item $F^{-1}(W)=U$.
\item $F(U \backslash S )= W\backslash T $. 
\item The {\defn{locus of instability}}, the set of points in $W$ such that $F|U_t$ is not-stable, is equal to $T$.
\end{enumerate}
We call $F$ an {\em{excellent unfolding}} if, in addition to the above, the following holds.
\begin{enumerate}
\setcounter{enumi}{3}
\item The $0$-stables are constant in the family. That is, there does not exist a curve of points in $W$ such that $F|U_t$ has a stable type with a zero-dimensional stratum. 
\end{enumerate}
\end{definition}

Now we need to look for invariants of $f_t$, the elements of the family, that will produce equisingularity. 
It is well-known that if $(X,x)$ is a complete intersection complex analytic set, then the complex link of $x$
is a wedge of spheres of real dimension $\dim _\C X -1$. 

Let $g:(\C ^{N+1},0)\to (\C ,0)$ be a complex analytic function.
If $H^i$ is a plane of dimension $i$ through the origin in $(\C ^{N+1},0)$, then in general $V(g)\cap H^i$ is a complete
intersection.
\begin{definition}[Cf.\ \cite{damon-highmult}]
The {\em{$k$th higher multiplicity}} is the number 
\[
\mu ^k(g) = \dim _\C H _{k-1}(\LL ^k; \C)
\]
where $\LL ^k$ is the complex link of $V(g)\cap H^{k+1}$ at $0$ and $1\leq k \leq N$.
\end{definition}
For sufficiently general $H^k$ this is a well-defined invariant of $V(g)$.

Now consider the family $f_t:(\C ^n ,0)\to (\C ^{n+1} ,0)$. The image of $f_t$ is a hypersurface and so we can define $g_t:(\C ^{n+1},0) \to (\C ,0)$ to be the reduced defining equation for this hypersurface.

We now state Corollary~6.13 from \cite{cntocn+1} which is in some sense a generalization of the classical Brian\c{c}on--Speder--Teissier result on the equisingularity of families of isolated hypersurface singularities.
\begin{proposition}
\label{equi-equiv-result}
Suppose that $f:(\C ^n,0)\to (\C ^{n+1},0)$ is a corank $1$ mono-germ with an isolated instability at $0$ and that $F$ is an excellent one-parameter unfolding of $f$ that is stratified by stable type outside of $S$ and $T$.

Then, the image of $F$ is Whitney equisingular along the parameter axis $T$ if and only if
the sequence $\left( \mu  ^1(g_t), \dots , \mu  ^{n}(g_t), \widetilde{\chi } (\mf (g_t)) \right) $ is constant in the family.
\end{proposition}

This statement is not very aesthetically pleasing. It would be improved if we could replace the mix of invariants - the higher multiplicities and the Euler characteristic - by a more consistently defined sequence. In the rest of this section we shall make this improvement for the `if' part of the statement, and in the $n=3$, $p=4$ case do so completely. Furthermore, in both situations, we shall do this for good unfoldings rather than the stronger notion of excellent.

First we shall describe our consistently defined set of invariants. 
Let $f:(\C ^n,0)\to (\C ^{n+1},0)$ be a corank $1$ map-germ with an isolated instability and let $\dis (f)$ be the disentanglement of $f$.
For an open and dense set $Z$ in the space of $i$-planes through the origin in
$\C ^{n+1} $, the set $\dis (f)\cap H$ is homotopically equivalent to a wedge of $(i-1)$-spheres
for $H\in Z$ and the number of spheres is independent of $H$ and the choice of 
disentanglement map. See Corollary 4.2 and Lemma 7.8 of \cite{damon-highmult}.

This leads to the following definition.
\begin{definition}
\label{good-image-milnor-defn}
The {\em{$k$th image Milnor number of $f$}}, denoted $\mu _I^k(f), $ is equal to $b_{k-1} (\dis (f) \cap H^k)$, where  $H^k$ is a generic $k$-plane through the origin in $\C ^{n+1}$. 

The {\em{$\mu _I^*$-sequence for the map $f$}} is the set $\mu _I^k (f)$ for $k=0,\dots , n+1$.
\end{definition}
Note that $\mu _I^{n+1}(f)=\mu _I (f)$,  i.e., it is the usual image Milnor number from \cite{vancyc}.

\begin{remark}
Unfortunately, definitions from my previous papers have rather ignominiously collided with each other. The definition given in \cite{cntocn+1} for the $k$th image Milnor number is that given above and is also denoted $\mu _I^k(f)$.

However, in \cite{calcgen} a totally different object was given the name the $k$th image Milnor number. This object is the number of spheres in the homotopy type of the $k$th disentanglement which we used in Section~\ref{sec:3to4} on maps from $3$-space to $4$-space. 
Fortunately, here and in \cite{calcgen} it was denoted $\mu _{I_k}$, which is sufficiently different to $\mu _I ^k$.

It is important to note that these concepts are different and therefore, in general,  $\mu _{I_k}(f)\neq \mu _I ^k(f)$ except for $\mu _{I_1}(f)=\mu _I^{n+1}(f)=\mu _I(f)$.

To remedy this collision of notation I shall in future papers use the definition in Definition~\ref{good-image-milnor-defn}.
\end{remark}

We need the next lemma to force the constancy, in the family, of the Milnor numbers of the multiple point spaces.
\begin{lemma}
\label{lem:mui-const}
Suppose that $f:(\C ^n,0)\to (\C ^{n+1},0)$ is corank $1$ map-germ with an isolated instability and $F$ is an excellent unfolding.   Then
$\mu _I(f_t)$ is constant if and only if $\mu \left( D^k(f_t,\PP )\right)$ is constant for all $D^k(\widetilde{f_t},\PP )$ that are non-empty.  
\end{lemma}
\begin{proof}
The number of points in $F^{-1}(0,t)$ is constant by condition (ii) of excellent unfolding. The result then follows from Corollary~4.7 of \cite{excellent} (where $F^{-1}(0,t)$ is denoted $s(f_t)$). 
\end{proof}

Our main equisingularity result is the following.
\begin{theorem}
\label{mainequithm}
Suppose that $f:(\C ^n,0)\to (\C ^{n+1},0)$ is a corank $1$ mono-germ with an isolated instability at $0$ and that $F$ is a good one-parameter unfolding of $f$ that is stratified by stable type outside of $S$ and $T$.

If the sequence $\left( \mu  _I^1(f_t), \dots , \mu  _I^{n+1}(f_t) \right) $ is constant in the family,
then the image of $F$ is Whitney equisingular along the parameter axis $T$.
\end{theorem}
\begin{proof}
As $F$ is good the number of points in $F^{-1}(0,t)$ is constant by condition (ii) of good unfolding. Hence, condition (iv) of excellent unfolding holds from Corollary~5.8 of \cite{excellent} and so $F$ is excellent.

Since $F$ is excellent the strata outside $T$ are two-dimensional. (Since we are only interested in strata close to $T$ we can assume that $U$ is chosen small enough to exclude any zero-dimensional strata.) We can Whitney stratify $f_t$ by taking the intersection of strata of $F$ with $\C ^n \times \{ t \} $, see \cite{polar} section~6, in particular Proposition~6.4. See also Section~6 of \cite{cntocn+1}. (The origins in $U_t$ and $W_t$ will automatically be Whitney strata.)

Since the stratification of $F$ (outside $S$ and $T$) is by stable type the stratification of the image of $f_t$ is by stable type and this is the canonical stratification. 
The strata in the disentanglement of $f_t$ correspond to these strata with the addition of zero-dimensional strata. Let $X$ be a stratum of $F$ and $X_t$ be the intersection with $\C ^n \times \{ t \} $. The transverse Milnor fibre for a non-zero-dimensional stratum $X_t$ is determined by the stratum $X$ in the stratification of $F$ and hence for each stratum we can consider the transverse Milnor fibre constant along $T$.

By Lemma~\ref{lem:mui-const}, the constancy of $\mu _I(f_t)$ implies that all the Milnor numbers of the multiple point spaces and their restriction to reflecting hyperplanes are constant. Corollary~\ref{linsum} thus shows that  $\chi \left( \cl{X_t}, \cl{X_t}\backslash X_t \right)$ is constant in the family.
For $0$-dimensional strata $X_t$ we have $\chi \left( \cl{X_t}, \emptyset \right)$ constant as these strata correspond to some zero-dimensional multiple point space.

Now, $\chi \left( \dis(f_t) \right) $ is constant along $T$ as it equals $1+(-1)^n\mu _I ^{n+1}(f_t)$ and therefore, by Theorem~\ref{dis-formula}, the Euler characteristic of the Milnor fibre of $g_t$ is constant. As $g_t$ is reduced with a singular set of codimension $2$ in the ambient space the Milnor fibre is connected by \cite{km} and so the reduced Euler characteristic of the Milnor fibre of $g_t$ is constant in the family.

Turning to the other terms in the sequence in Proposition~\ref{equi-equiv-result}, we have that, by a generalization of Section~4 of \cite{damon-highmult},  $\mu ^i(g_t) = \mu _I ^i(f_t)$ for all $i\leq n$. Hence by assumption $\mu ^i(g_t)$ is constant for all $1\leq i \leq n$.

Thus, by Proposition~\ref{equi-equiv-result}, the image of $F$ is Whitney equisingular along $T$.
\end{proof}

We can now generalize part of Theorem~5.2 of \cite{diswe} (see also Conjecture~5.6 there)  and produce a converse to our Theorem~\ref{mainequithm} in the case of maps from $(\C ^3,0)$ to $(\C ^4,0)$.
\begin{theorem}
Suppose that $f:(\C ^3,0)\to (\C ^4,0)$ is a corank $1$ mono-germ with an isolated instability at $0$ and that $F$ is a good one-parameter unfolding of $f$ that is stratified by stable type outside of $S$ and $T$.

The sequence $\left( \mu  _I^1(f), \mu  _I^2(f), \mu  _I^3(f), \mu  _I^4(f)  \right) $ is constant in the family if and only if the image of $F$ is Whitney equisingular along the parameter axis $T$.
\end{theorem}
\begin{proof}
The `only if' part follows from Theorem~\ref{mainequithm}. For the converse suppose that the image of $F$ is Whitney equisingular along $T$.

Since $T$ is a Whitney stratum there cannot exist a curve arising from zero-dimensional singularities in the image of $f_t$ causing another $1$-dimensional stratum to meet at the origin. Therefore, as $F$ is good, it must also be excellent.

Next, $\left( \mu ^1 (g_t), \dots , \mu _n(g_t), \widetilde{\chi }  \left( \mf (g_t) \right) \right) $ is constant by Proposition~\ref{equi-equiv-result} and as stated in the previous theorem $\mu ^i(g_t)=\mu ^i_I(f_t)$ for all $1\leq i\leq n$.

Corollary~\ref{c3toc4cor} shows that we can calculate $\chi (\mf (g_t))$ in terms of $\notemp _2$, $\notemp _3$, 
$ \mu (D^2) $, $\mu (D^2|H) $, $\mu (D^3)$, $\mu (D^3|H_1)$ and $Q$ for each $f_t$. 

The fact that $D^k(\widetilde{f} _t) \neq \emptyset $ for $t\neq 0$ implies that $D^k(\widetilde{f} _0) \neq \emptyset $ is shown in the proof of Theorem~4.3 of \cite{excellent}. This fact in turn implies that $\notemp _k$ is upper semi-continuous.

Suppose that $\notemp _2$ increases at $t=0$. Then $D^2(\widetilde{f} _0)=\emptyset $, i.e., $\widetilde{f}_0$ is not singular and has no ordinary double points. This means that $f_0$ is non-singular. But if $f_0$ is non-singular, then so is any nearby deformation,  i.e., $D^2(\widetilde{f} _t)=\emptyset $ for all $t$. This contradicts an increase in $\notemp _2$. Hence $\notemp _2$ is constant.

As the Milnor number of an isolated complete intersection is upper semi-continuous (\cite{looijenga} p126) we deduce that all the other terms in the right-hand side of Corollary~\ref{c3toc4cor} are constant in the family. From Lemma~\ref{lem:mui-const} we can deduce that $\mu _I^{n+1}(f_t)=\mu _I(f_t)$ is constant.
\end{proof}
\begin{remark}
Similar reasoning gives a proof of the same statement for the $\C ^2$ to $\C ^3$ case, a multi-germ version of which is given in \cite{diswe}. 
\end{remark}

\section{Concluding remarks}
We have seen that Theorem~\ref{maineulerthm} generalizes work of Massey and Siersma, answers a long-standing question involving the relationship between the disentanglement and the Milnor fibre of the image of a map-germ from $(\C ^n,0)$ to $(\C ^{n+1},0)$, and solves a significant problem in equisingularity.

A number of avenues remain to be explored:

One reason we could not, in general, reverse the implication in Theorem~\ref{mainequithm} to get that  $\mu _I^*$ constant implies Whitney equisingularity of the image is because we do not know the Euler characteristic of the Milnor fibre of the image of a stable corank $1$ mono-germ from $(\C ^n,0)$ to $(\C ^{n+1},0)$. This is because this number is a crucial ingredient in Theorem~\ref{dis-formula} which is used in the proof of Corollary~\ref{c3toc4cor} for the $\C ^3$ to $\C ^4$ case.

For a minimal corank $1$ stable map from $\C ^n$ to $\C ^{n+1}$ it should be noted that as $n$ gets larger the number of monomials in the defining function for the image becomes explosively larger. It may be worth observing that using the techniques in \cite{mond-pell} we can write the function as the determinant of a matrix. 
Also we may be able to use Massey's work on L\^e numbers. It may be possible to just calculate one L\^e number and use induction. This is because for the image of $f$ we can restrict to a generic hyperplane through the origin. This allows us to calculate one L\^e number that relates the Euler characteristic of the Milnor fibre of the image to the Euler characteristic of the Milnor fibre of the restriction of the image to the hyperplane. This latter can be calculated since the image is the image of an $\AA _e$-codimension one map. Its disentanglement is homotopically equivalent to a single sphere. The Euler characteristic of each stratum is easy to compute using the multiple point spaces and by the inductive hypothesis we know the Euler characteristic of transverse Milnor fibres arising from mono-germs.
We would also need to know the precise relationship in Theorem~\ref{linsum} but this appears to be just a combinatorial issue.

Another point not mentioned is that we would like the map $F$ to be equisingular along $S$ and $T$ rather than just the image of $F$ equisingular along $T$. Gaffney has sent the author a personal communication outlining how for corank $1$ maps from $(\C ^n,0)$ to $(\C ^{n+1},0)$ the equisingularity of the image of $F$ along $T$ automatically implies the equisingularity of $F$ along $S$ and $T$. However, this has yet to be published and hence Theorem~\ref{mainequithm} has not included this result in the statement. A simple proof of his result would be welcomed.

\end{document}